\documentclass{siamltex}
\usepackage[dvips]{graphicx}
\usepackage{amsmath}
\usepackage{latexsym}
\usepackage{amssymb}
\usepackage{color}
\usepackage{algorithm,algorithmic}

\newtheorem{remark}{Remark}[section]

\makeatletter
\newcommand{\ssymbol}[1]{^{\@fnsymbol{#1}}}
\makeatother

\newtheorem{exe}{Exemple}
\usepackage[mathscr]{euscript}
\title{Tensor extrapolation methods with applications 	 }
\author{ F.P.A Beik \thanks{Department of Mathematics, Vali-e-Asr University of Rafsanjan,\ P.O. Box 518, Rafsanjan, Iran} \and A. El Ichi \footnotemark[2]  \and K. Jbilou \thanks{LMPA, 50 rue F. Buisson, ULCO Calais, France; jbilou@univ-littoral.fr } \and R. Sadaka\thanks{Department of Mathematics University Mohammed V Rabat, Morocco}}

\date{}

\begin{document}

	\maketitle



\begin{abstract}
In this paper, we mainly develop the well-known vector and matrix polynomial extrapolation methods in
tensor framework. To this end, some new products between tensors are defined and the concept of positive definitiveness
is extended for tensors corresponding to T-product. Furthermore, we discuss on the solution of least-squares problem associated with
a tensor equation using Tensor Singular Value Decomposition (TSVD). Motivated by the effectiveness of proposed vector extrapolation method in [Numer. Algorithms,  51 (2009), 195--208], we describe how an extrapolation technique can be also implemented on the sequence of tensors  produced by truncated TSVD (TTSVD) for solving possibly ill-posed tensor equations.\\
\end{abstract}



\medskip

	

\section{Introduction}

In the last few years,  several iterative methods   have been proposed for solving large and sparse  linear and nonlinear systems of equations. When an iterative process converges slowly, the
extrapolation methods are required to obtain rapid convergence. The  purpose of vector extrapolation methods is  to transform a sequence of vector or matrices generated by some process to a new one  that  converges faster than the initial sequence.   The well known extrapolation methods  can be  classified into two categories, the polynomial methods that  includes   the minimal polynomial extrapolation (MPE) method of Cabay and Jackson \cite{B101}, the modified minimal  polynomial extrapolation (MMPE) method of Sidi, Ford ans Smith \cite{B0301}, the reduced rank extrapolation (RRE) method of Eddy \cite{B03} and Mesina \cite{B3011},  Brezinski \cite{B0032}  and Pugatchev \cite{B132}, and  the $\epsilon$-type algorithms including the topological $\epsilon$-algorithm  of Brezinski \cite{B0032} and the vector $\epsilon$-algorithm of Wynn \cite{Wynn}. \\
Efficient implementations of some of these extrapolation methods have been proposed by Sidi \cite{B03005}  for the RRE and MPE  methods using QR decomposition while Jbilou and sadok \cite{B350} gives   an efficient implementation of the MMPE based on a LU  decomposition with pivoting strategy. It was also shown  that when applied to linearly generated vector sequences, RRE  and TEA methods are mathematically equivalent to GMRES and Lanczos methods, respectively.  Those results were also extended to the block and global cases when dealing with matrix sequences, see \cite{B2071,B360}. Our aim in this paper is to define the analogue of these vector and matrix extrapolation methods to the tensor framework.\\
\noindent Basically, in the present paper, we develop some tensor extrapolation methods namely, the Tensor RRE (TRRE), the Tensor MPE (TMPE), the Tensor MMPE (TMMPE) and the Tensor Topological $\epsilon$-Algorithm (TTEA). We give some properties and show how these new tensor extrapolation methods can be applied to sequences obtained by truncation of the Tensor Singular Value Decomposition (TSVD) when applied to linear tensor discrete ill-posed problems.\\

\noindent The remainder of this paper is organized as follows.  Before ending this section, we recall some fundamental concepts in tensor framework.  In Section 2, we  give notations, some basic definitions and properties related to tensors. Moreover, we introduce the concept of positive definiteness for tensors with respect to T-product, some new products are also defined between tensors and their properties are analyzed. In Section 3, we introduce the    tensor versions of the vector polynomial extrapolation methods namely the   Tensor  Reduced Rank Extrapolation (TRRE), the Tensor  Minimal Polynomial Extrapolation (TMPE) and the   Tensor Modified Minimal Polynomial Extrapolation (TMMPE),
the Tensor Topological  $\epsilon$-Algorithm (TTEA).
Section 4  describes the TSVD, the truncated or low rank  version of TSVD and shows how to apply  the Tensor Reduced Rank Extrapolation method together with the truncated TSVD, to the solution of linear discrete tensor ill-posed problems.


\noindent \textit{Preliminaries:} A tensor is  a multidimensional array of data. The number of indices of a tensor is called modes or ways. 
Notice that a scalar can be regarded as a zero mode tensor, first mode tensors are vectors and matrices are second mode tensor.
For a given N-mode tensor $ \mathscr {X}\in \mathbb{R}^{n_{1}\times n_{2}\times n_{3}\ldots \times n_{N}}$, the notation $x_{i_{1},\ldots,i_{N}}$ (with $1\leq i_{j}\leq n_{j}$ and $ j=1,\ldots N $) stand for the element $\left(i_{1},\ldots,i_{N} \right) $ of the tensor $\mathscr {X}$. The norm of a tensor $\mathscr{X}\in \mathbb{R}^{n_1\times n_2\times \cdots \times n_\ell}$ is specified by
\[
\left\| \mathscr{X} \right\|^2 = {\sum\limits_{i_1  = 1}^{n_1 } {\sum\limits_{i_2  = 1}^{n_2 } {\cdots\sum\limits_{i_\ell = 1}^{n_\ell} {x_{i_1 i_2 \cdots i_\ell }^2 } } } }^{}.
\] Corresponding to a given tensor $ \mathscr {X}\in \mathbb{R}^{n_{1}\times n_{2}\times n_{3}\ldots \times n_{N}}$, the notation $$ \mathscr {X}_{\underbrace{::\ldots:}_{(N-1)-\text{ times}}k}\; \; {\rm for }  \quad k=1,2,\ldots,n_{N}$$ denotes a tensor in $\mathbb{R}^{n_{1}\times n_{2}\times n_{3}\ldots \times n_{N-1}}$ which is obtained by fixing the last index and is called frontal slice. Fibers are the higher-order analogue of matrix rows and columns. A fiber is
defined by fixing all the indexes  except  one. A matrix column is a mode-1 fiber and a matrix row is a mode-2 fiber. Third-order tensors have column, row and tube fibers. An element $c\in \mathbb{R}^{1\times 1 \times n}$ is called a tubal--scalar of length $n$ \cite{B029}. More details are found in  \cite{B129,B031}.  \\

\begin{figure}[!h]
	\begin{center}
		\caption{  (a) Frontal, (b) horizontal, and (c) lateral slices of a third   order tensor.   (d) A  mode-3 tube fibers. }
		\label{fig:fibre111}
	\end{center}
\end{figure}

		 \section{Definitions and new tensor products}\label{sec:section2}
		
		 The current section is concerned with two main parts. In the first part, we recall definitions and properties related to T-product. Furthermore, we develop the definition of positive definiteness for tensors and establish some basic results. The second part deals with presenting some new products between tensors which can be used for simplifying the algebraic computations of the main results.
		
		 \subsection{Definitions and properties}
		 In this part, we briefly review some concepts and notations related to the T-Product, see \cite{C220,B010,B029,B0290} for more details. 	
		 \begin{definition}
		 	The \textbf{T-product} ($\ast$) between  two tensors
		 	$\mathscr {X} \in \mathbb{R}^{n_{1}\times n_{2}\times n_{3}} $ and $\mathscr {Y} \in \mathbb{R}^{n_{2}\times m_{2}\times n_{3}} $ is an  ${n_{1}\times m_{2}\times n_{3}}$ tensor  given by:	
		 	$$\mathscr {X}\ast\mathscr {Y}={\rm Fold}({\rm bcirc}(\mathscr {X}){\rm MatVec}(\mathscr {Y}) )  $$
		 	where
		 	\[{\rm bcirc}(\mathscr {X})=\left( {\begin{array}{*{20}{c}}
		 		{{X_1}}&{{X_{{n_3}}}}&{{X_{{n_{3 - 1}}}}}& \ldots &{{X_2}}\\
		 		{{X_2}}&{{X_1}}&{{X_{{n_3}}}}& \ldots &{{X_3}}\\
		 		\vdots & \ddots & \ddots & \ddots & \vdots \\
		 		{{X_{{n_3}}}}&{{X_{{n_{3 - 1}}}}}& \ddots &{{X_2}}&{{X_1}}
		 		\end{array}} \right)\in \mathbb{R}^{n_{1}n_{3}\times m_{2}n_{3}}\]
		 	$${\rm MatVec}(\mathscr {Y} ) = \begin{pmatrix}
		 	Y_{1}   \\
		 	Y_{2}   \\
		 \vdots \\
		 	Y_{n_{3}}\end{pmatrix} \in \mathbb{R}^{n_{2}n_{3}\times m_{2}} ,  \qquad {\rm Fold }({\rm MatVec}(\mathscr {Y}) ) =  \mathscr {Y}$$
		 \end{definition}
		 here  for $i=1,\ldots,n_3$, $X_i$ and $Y_i$ are frontal slices of the tensors $\mathscr {X}$ and $\mathscr {Y}$, respectively.\\
		 The $n_{1} \times n_{1}\times n_{3}$ identity tensor $\mathscr{I}_{ n_{1}  n_{1}  n_{3}}$ is the tensor whose first frontal slice is the $ n_{1} \times n_{1}$ identity matrix, and whose other frontal slices are all zeros, that is
		 $${\rm MatVec}(  {\mathscr{I}_{ n_{1}n_{1}n_{3}}})  = \begin{pmatrix}
		 I_{n_{1}n_{1}}   \\
		 0 _{n_{1}n_{1}} \\
		 \vdots \\
		 0_{n_{1}n_{1}}\end{pmatrix}$$
		 where $I_{n_{1}n_{1}}$ is the identity matrix. \\
		 In the special case in which $n_1=1$, the identity tensor is a tubal--scalar and denoted by ${\rm \bf e}$,       in other words ${\rm MatVec}({\rm \bf e})  =(1,0,0\ldots,0)^T$.\\
		
		 \begin{definition} We have the following definitions
		 	
		 	\begin{enumerate}
		 		\item An $n_{1}\times n_{1} \times n_{3}$ tensor $\mathscr{A}$ is invertible, if there exists a tensor $\mathscr{B}$ of order  $n_{1}\times n_{1} \times n_{3}$  such that
		 		$$\mathscr{A}\ast \mathscr{B}=\mathscr{I}_{ n_{1}  n_{1}  n_{3}} \qquad \text{and}\qquad \mathscr{B}\ast \mathscr{A}=\mathscr{I}_{ n_{1}  n_{1}  n_{3}}$$
		 		It is clear that 	$\mathscr{A}$ is invertible if and only if   ${\rm bcirc}(\mathscr{A})$ is invertible (see \cite{B30112}).
		 		\item  If $\mathscr{A} $ is  an $n_{1} \times n_{2}\times n_{3}$ tensor, then $\mathscr{A}^{T}$ is the  $n_{2} \times n_{1}\times n_{3}$   tensor obtained
		 		by transposing each of the front-back frontal slices and then reversing the order of transposed frontal slices 2 through $n_{3}$.
		 	\end{enumerate}
		 \end{definition}
		
		 \noindent \begin{exe}
		 	If $\mathscr{A} \in \mathbb{R}^{n_{1}\times n_{2}\times 5}$ and its frontal slices are given by the $n_{1}\times n_{2}$
		 	matrices $A_{1}, A_{2}, A_{3}, A_{4}, A_{5}$, then
		 	$$\mathscr{A}^{T} ={\rm Fold} \begin{pmatrix}
		 	A_{1}^{T}   \\
		 	A_{5}^{T} \\
		 	A_{4}^{T} \\
		 	A_{3}^{T}\\
		 	A_{2}^{T} \\
		 	\end{pmatrix} $$
		 	
		 \end{exe}

\begin{remark}\label{rem1.1new}
	As pointed out earlier the tensor $\mathscr{A}$ of order $m\times m \times n_{}$ is invertible iff ${\rm bcirc}(\mathscr{A})$ is invertible. It is equivalent to say that  $\mathscr{A}$ is invertible iff
	 $\mathscr{A} \ast \mathscr{X}=\mathscr{O}$ implies $\mathscr{X}=\mathscr{O}$ where $\mathscr{X}\in \mathbb{R}^{m\times 1 \times n}$ and $\mathscr{O}$ is zero tensor.
\end{remark}

\noindent Here we define the notion of positive definiteness for tensors in term of T-product which can be seen as a natural extension of the same concept for matrices.

\begin{definition}
The tensor $\mathscr{A}\in \mathbb{R}^{m\times m \times n_{}}$
is said to be positive (semi) definite if $$(\mathscr{X}^T \ast \mathscr{A} \ast \mathscr{X})_{::1} > (\ge ) 0,$$
for all nonzero tensors $\mathscr{X}\in \mathbb{R}^{m\times 1 \times n}$.
\end{definition}

\begin{remark}\label{rem1.2new}
In view of Remark \ref{rem1.1new}, it is immediate to see that every positive definite tensor is invertible. For any tensor $B\in \mathbb{R}^{n_1\times n_2 \times n_3}$ and arbitrary nonzero tensor $\mathscr{X}\in \mathbb{R}^{n_2\times 1 \times n_3}$,  it can be seen that $$(\mathscr{X}^T \ast \mathscr{B}^T\ast  \mathscr{B} \ast \mathscr{X})_{::1}=((\mathscr{B} \ast \mathscr{X})^T\ast  \mathscr{B} \ast \mathscr{X})_{::1}=\|\mathscr{B} \ast \mathscr{X}\|^2 \ge 0,$$	
	in which the last equality follows from \cite{B0290}. As a result, the tensor $\mathscr{B}^T\ast \mathscr{B}$
	is positive semi definite. Evidently, for any scalar $\epsilon > 0$,
	 $$(\mathscr{X}^T \ast (\mathscr{B}^T\ast \mathscr{B}+\epsilon~ \mathscr{I}_{n_2n_2n_3}) \ast \mathscr{X})_{::1}=\|\mathscr{B} \ast \mathscr{X}\|^2+\epsilon  \|\mathscr{X}\|^2>0,$$
	This shows that the tensor  $\mathscr{B}^T\ast \mathscr{B}+\epsilon~ \mathscr{I}_{n_2n_2n_3}$ is positive definite.
\end{remark}
		 \medskip
		 \begin{definition}
		 	Let $\mathscr{A}\in \mathbb{R}^{n_1\times n_2 \times n_3}$. Then the tensor $ \mathscr{X}\in  \mathbb{R}^{n_2\times n_1 \times n_3}$ satisfying the following four conditions\\
		 	\begin{tabular}{ccccc}
		 		{\rm (a)} &$\mathscr{A}\ast \mathscr{X} \ast \mathscr{A}=\mathscr{A}$ && {\rm (b)}& $\mathscr{X}\ast \mathscr{A} \ast \mathscr{X}=\mathscr{X}$,\\
		 		{\rm (c)} &$(\mathscr{A} \ast \mathscr{X})^T=\mathscr{A} \ast \mathscr{X}$, &&
		 		{\rm (d)} & $(\mathscr{X}\ast \mathscr{A})^T = \mathscr{X}\ast \mathscr{A}$.
		 	\end{tabular}\\
		 	is called the Moore-Penrose inverse of $\mathscr{A}$ and denoted by $\mathscr{A}^\dagger$.
		 \end{definition}   		
		 \medskip		
		 \noindent In view of \cite[Lemmas 3.3 and 3.16]{B0290} and by using straightforward computations, one can easily observe that conditions (a)-(d) determine $\mathscr{A}^\dagger$ uniquely. Here we further comment that if $\mathscr{A}$ is invertible, then $\mathscr{A}^\dagger =\mathscr{A}^{-1}$.\\
		 For  $\mathscr{X},\mathscr{Y}$ two tensors in $\mathbb{R}^{n_{1}\times 1\times n_{3}}$,
		 theT-scalar product $\left\langle {.,.} \right\rangle$
		 is a bilinear form   
		 defined by:
		 \begin{align}\label{bilinearform3d}
		 \begin{cases}
		 \mathbb{R}^{n_{1}\times 1\times n_{3}}\times \mathbb{R}^{n_{1}\times 1\times n_{3}}   & \longrightarrow  \mathbb{R}^{1\times 1\times n_{3}} \\
		 \qquad  \qquad (\mathscr{X},    \mathscr{Y})\qquad   &\longrightarrow  \langle \mathscr{X}, \mathscr{Y} \rangle =  \mathscr{X}^{T}\ast \mathscr{Y}
		 \end{cases}.
		 \end{align}
		 Let  $\mathscr{X}_{1},\ldots, \mathscr{X}_{\ell}$ a collection of $\ell$ third tensors in
		 $\mathbb{R}^{n_{1}\times 1\times n_{3}}$, if
		 \begin{align*}
		 \left\langle {\mathscr{X}_{i}, \mathscr{X}_{j}} \right\rangle =
		 \begin{cases}
		 \alpha_{i}{\rm \bf e}&i= j \\
		 0&i\neq j
		 \end{cases}.
		 \end{align*}
		 where  $\alpha_{i}$ is a non-zero scalar, then the set $\mathscr{X}_{1},\ldots, \mathscr{X}_{\ell}$  is said to be an  orthogonal collection of tensors.  The collection is called orthonormal if $\alpha_{i}=1$, $i=1,\ldots,l$.

		 \begin{definition}
		 	An $n\times n \times \ell$ real-values tensor $\mathscr{Q}$ is said to be is orthogonal if $\mathscr{Q}^T \ast \mathscr{Q} =\mathscr{Q} \ast \mathscr{Q}^T =\mathscr{I}_{nn\ell}$.
		 \end{definition}

		 \noindent We end this part with the following proposition.
		
		 \begin{proposition}\label{prop2.1}
		 	Suppose that $\mathscr{A}$ and $\mathscr{B}$ are two tensors of order $n_1\times n_2 \times n_3$. Then
		 	\[
		 	(\mathscr{A}^T\ast \mathscr{B})_{ij:} =(\mathscr{A}(:,i,:))^T \ast \mathscr{B}(:,j,:).
		 	\]
		 \end{proposition}
		 \begin{proof}
		 	Let $\widetilde{\mathscr{I}}_\tau$ be  an $n_2\times 1 \times n_3$ tensor
		 	whose all frontal slices are zero except its first frontal being the $\tau$-th column of the $n_2\times n_2$ identity matrix. Then we have
		 	\[
		 	\mathscr{B} \ast \widetilde{\mathscr{I}}_j = \mathscr{B}(:,j,:)\qquad
		 	\text{and}
		 	\qquad
		 	\widetilde{\mathscr{I}}_i^T\ast  \mathscr{A}^T=\mathscr{A}^T(i,:,:)=(\mathscr{A}(:,i,:))^T.
		 	\]
		 	Notice also  that
		 	\[
		 	(\mathscr{A}^T\ast \mathscr{B})_{ij:}=	\widetilde{\mathscr{I}}_i^T\ast (\mathscr{A}^T\ast \mathscr{B})\ast \widetilde{\mathscr{I}}_j.
		 	\]
		 	From \cite[Lemma 3.3]{B0290}, it is known that
		 	\[
		 	\tilde{\mathscr{I}}_i^T\ast (\mathscr{A}^T\ast \mathscr{B})\ast \tilde{\mathscr{I}}_j = (\tilde{\mathscr{I}}_i^T\ast \mathscr{A}^T)\ast (\mathscr{B} \ast \tilde{\mathscr{I}}_j),
		 	\]
		 	which completes the proof. 
		 \end{proof}

		 \subsection{New tensor  products }
		 In order to simplify derivation of generalized extrapolation methods in tensor format, we need to define  new tensor products
		 .
		 \begin{definition} \label{def2.3}
		 	Let $\mathscr{A}$ and $\mathscr{B}$ be 4-mode tensors with frontal slices $\mathscr{A}_{i}\in \mathbb{R}^{n_1\times n_2 \times n_3}$ and $\mathscr{B}_{j}\in \mathbb{R}^{n_1\times n_2 \times n_3}$  for $i=1,2,\ldots,\ell$ and $j=1,2,\ldots,k$, respectively. The product $\mathscr{A} \diamondsuit  \mathscr{B} $ is defined as a 5-mode tensor of order $n_2 \times n_2  \times n_{3} \times k \times \ell$ defined as follows:
		 	$$(\mathscr{A}\diamondsuit  \mathscr{B})_{:::ji}  = \mathscr{A}_i^T \ast \mathscr{B}_j,
		 	$$
		 	for $i=1,2,\ldots,\ell$ and $j=1,2,\ldots,k$. In the case where  $k=1$, i.e. $\mathscr{B}\in \mathbb{R}^{n_1\times s \times n_3}$, $\mathscr{A} \diamondsuit  \mathscr{B} $ is a 4-mode tenor whose $i$-th frontal slice is given by $\mathscr{A}_i^T \ast \mathscr{B}$ for $i=1,2,\ldots,\ell$.
		 \end{definition}
		
		 \medskip
		 \noindent Here we comment that the $\diamondsuit$	product can be seen as a generalization of $\diamond$-product between two matrices given in \cite{B220}. In the sequel, we further present an alternative product called $\star$-product which can be seen as an extension
		 of the $\ast$-product between a set of matrices and vectors; see \cite{Jbilou} for more details.\\
		 \medskip
		
		 \begin{definition}
		 	Let $\mathscr{A}$ be a 5-mode tensor of order $n_2 \times n_1  \times n_{3} \times k \times \ell$  and $\mathscr{B}$ be a 4-mode tensor with frontal slices $\mathscr{B}_1,\ldots,\mathscr{B}_k\in \mathbb{R}^{n_1\times n_2 \times n_3}$. The product $\mathscr{A} \star \mathscr{B}$ is defined as a 4-mode tensor of order
		 	$n_2 \times n_2  \times n_{3} \times \ell$ such that
		 	\[
		 	(\mathscr{A} \star \mathscr{B})_{:::i} =  \sum_{j=1}^{k} \mathscr{A}_{:::ji} \ast \mathscr{B}_j,\quad i=1,2,\ldots,\ell.
		 	\]
		 	As a natural way, for the case that $\mathscr{A}$ is a 4-mode tensor of order $n_2 \times n_1  \times n_{3} \times \ell$ then $\mathscr{A} \star \mathscr{B}\in \mathbb{R}^{n_2 \times n_2  \times n_{3}}$ defined by  $\mathscr{A} \star \mathscr{B}= \sum_{\eta=1}^{k} \mathscr{A}_{\eta} \ast \mathscr{B}_\eta$ where $\mathscr{A}_{i}$ stands for the $i$-th frontal slice of $\mathscr{A}$ for $i=1,2,\ldots,\ell$.
		 \end{definition}
		
		 \medskip
		 \noindent
		 In the main results, it is worth to define a notion
		 of the left inverse of a 5-mode tensors dealing with $\ast$, $\star$ and $\diamondsuit$ products.
		 To this end, we define the $\bar{\star}$-product
		 between two 5-mode tensors which can be  seen  as an extension of $\star$ product.\\
		
		 \medskip
		
		 \begin{definition}
		 	Let $\mathscr{A}\in \mathbb{R}^{n_1\times n_2 \times n_3 \times k \times \ell}$
		 	and $\mathscr{B}\in \mathbb{R}^{n_2\times n_1 \times n_3 \times k \times \ell}$, the product  $\mathscr{A}~ \bar{\star}~ \mathscr{B}$ is a 5-mode tensor of order $n_1\times n_1 \times n_3 \times k \times k$ such that for $\tau,\eta=1,2,\ldots,k$,
		 	\[
		 	(\mathscr{A}~ \bar{\star}~ \mathscr{B})_{:::\tau \eta} = \sum_{j=1}^{\ell} \mathscr{A}_{:::\eta j} \ast \mathscr{B}_{:::\tau j}.
		 	\]
		 \end{definition}
		
		 \begin{definition}\label{def2.7}
		 	The tensor $\mathscr{B}^{+}\in \mathbb{R}^{n_1\times n_2 \times n_3 \times k \times \ell}$ is called a left inverse of $\mathscr{B}\in \mathbb{R}^{n_2\times n_1 \times n_3 \times k \times \ell}$, if
		 	\begin{align*}
		 	(\mathscr{B}^{+}~ \bar{\star}~ \mathscr{B})_{:::\tau \eta} =
		 	\begin{cases}
		 	\mathscr{I}_{n_1n_1n_3}  &\tau = \eta \\
		 	\mathscr{O}   &\tau \ne \eta
		 	\end{cases}
		 	\end{align*}
		 	for $\tau,\eta=1,2,\ldots,k$. Here $\mathscr{O} $ is the 3-mode zero tensor of order $n_1\times n_1\times n_3$.
		 \end{definition}
		
		 \medskip
		 \noindent Now we establish a proposition which reveals the relation between the two  proposed products  $\star$ and $\bar{\star}$.
		
		 \begin{proposition}
		 	Let $\mathscr{A}\in \mathbb{R}^{n_1\times n_2 \times n_3 \times k \times k}$, $\mathscr{B}\in \mathbb{R}^{n_2\times n_1 \times n_3 \times k \times k}$ and
		 	$\mathscr{Y} \in \mathbb{R}^{n_1 \times n_2 \times  n_3 \times k}$. Then, the following relation holds
		 	\[
		 	(\mathscr{A}^{}~ \bar{\star}~ \mathscr{B})\star \mathscr{Y} =\mathscr{A}^*\star  (\mathscr{B}\star \mathscr{Y}).
		 	\]
		 	Here $\mathscr{A}^*$ stands for an 5-mode tensor of order ${n_1\times n_2 \times n_3 \times k \times k}$ associated with  $\mathscr{A}$ where $\mathscr{A}^*_{:::ij}=\mathscr{A}_{:::ji}$ for $1\le i,j\le k$.
		 \end{proposition}
		 \begin{proof}
		 	It is clear that both side of the above relation are 4-mode tensors of order $n_1\times n_2\times n_3 \times k$. To prove the assertion, we show that the frontal slices of both sides are equal. Let $1\le z \le k$, we can observe that
		 	\begin{eqnarray*}
		 		\left((\mathscr{A}^{}~ \bar{\star}~ \mathscr{B})\star \mathscr{Y}\right)_{:::z} & = &
		 		\sum_{\ell=1}^{k} (\mathscr{A}^{}~ \bar{\star}~ \mathscr{B})_{:::\ell z} \ast \mathscr{Y}_{\ell} \\
		 		& = &  \sum_{\ell=1}^{k}  \sum_{\mu=1}^{k} \mathscr{A}_{:::z \mu} \ast \mathscr{B}_{:::\ell \mu} \ast \mathscr{Y}_{\ell} \\
		 		& = & \sum_{\mu=1}^{k} \mathscr{A}_{:::z \mu} \ast \left(\sum_{\ell=1}^{k} \mathscr{B}_{:::\ell \mu} \ast \mathscr{Y}_{\ell} \right)\\
		 		& = & \sum_{\mu=1}^{k} \mathscr{A}_{:::z \mu} \ast  \left(\mathscr{B} \star \mathscr{Y}_{} \right)_{:::\mu}\\
		 		& = &  \sum_{\mu=1}^{k} \mathscr{A}^*_{:::\mu z} \ast  \left(\mathscr{B} \star \mathscr{Y}_{} \right)_{:::\mu}= ( \mathscr{A}^* \star (\mathscr{B} \star \mathscr{Y}_{}))_{:::z}.
		 	\end{eqnarray*}
		 The result follows immediately from the above computations.
		 \end{proof}
		

\noindent  We end this section by a proposition which can be established by using straightforward
		 algebraic computations.
		
		 \begin{proposition} \label{diamantpropos}
		 	Let $\mathscr{A},\mathscr{B},\mathscr{C}\in {\mathbb R}^{n_{1}\times s \times n_{3}\times \ell}$ and $\mathscr{D}\in {\mathbb R}^{s\times n_2 \times n_{3}}$. The following statements hold:
		 	\begin{enumerate}
		 		\item $ (\mathscr{A}+\mathscr{B})\diamondsuit \mathscr{C} = \mathscr{A}\diamondsuit \mathscr{C} +\mathscr{B}\diamondsuit \mathscr{C}$
		 		
		 		\item
		 		$\mathscr{A}\diamondsuit(\mathscr{B}+ \mathscr{C}) = \mathscr{A}\diamondsuit \mathscr{B} +\mathscr{A}\diamondsuit \mathscr{C}$

		 		\item $(\mathscr{A}\diamondsuit \mathscr{B})\star \mathscr{D}=\mathscr{A} \diamondsuit (\mathscr{B} \star \mathscr{D})$.		
		 		%
		 		
		 	\end{enumerate}

		 \end{proposition}

\section{Extrapolation methods based on tensor formats }\label{section3}
In this section, we define new tensor extrapolation methods. In the first part, we present in the tensor polynomial-type extrapolation methods by using the new tensor products introduced in the preceding section. In the second part, a tensor topological $\epsilon$-algorithm is developed. We notice that when we are dealing with vectors, all these new methods become just the classical vector extrapolation methods.

\subsection{Tensor polynomial-type extrapolation methods}\label{sub3.1} Corresponding to a given  sequence of tensors $(\bar{S}_{n})$ in $\mathbb{R}^{n_{1}\times n_2 \times n_{3} }$, we consider the transformation $(\bar{T}_{k})$ defined by
		\begin{align}\label{Tnk}
		{T}_{k}( {S}_{n} )={T} _{k}^{(n)}:=
		\begin{cases}
	\mathbb{R}^{n_{1}\times n_2\times n_{3}}    &\to \mathbb{R}^{n_{1}\times n_2\times n_{3} } \\
		\quad {S}_{n}  &\mapsto {T} _{k}^{(n)}=  {S}_{n}+{\mathscr{G}}_{k,n} \star \alpha^{}_k
		\end{cases}
		\end{align}
		where the 4-mode ${\mathscr{G}}_{k,n}$  with frontal slices ${G}_{i}(n)  \in \mathbb{R}^{n_{1}\times n_2\times n_{3} }$ is given for $i=1,\ldots,k$. The 4-mode tensor $\alpha^{}_k$ is unknown whose frontal slices are denoted by $\alpha^{(n)}_{i} \in \mathbb{R}^{n_2\times n_2\times n_{3}}$ for $i=1,\ldots,k$. As the vector and matrix case, in extrapolation methods, we aim to determine  the unknown tensors. To this end, we use the transformation $\widetilde{T}_{k}^{(n)}$  obtained from ${T} _{k}^{(n)}$ as follows:
		\begin{align*}
		\widetilde{T}_{k}^{(n)}=\widetilde{T}_{k}(  S_{n} )=   S_{n+1}+   {\mathscr{G}}_{k,n+1} \star \alpha^{}_k,
		\end{align*}
		here, the 4-mode ${\mathscr{G}}_{k,n+1}$  has frontal slices $ {G}_{i}(n+1)  \in \mathbb{R}^{n_{1}\times n_2\times n_{3} }$ for $i=1,\ldots,k$.
		Let $\Delta$ denote the forward difference operator on the index $n$ such that $\Delta {S}_{n}=  {S}_{n+1}-{S}_{n}$ and $\Delta {\mathscr{G}}_{k,n} $ stand for
		the 4-mode tensor whose frontal slices are given by $\Delta {G}_{i}(n)={G}_{i}(n+1)-{G}_{i}(n)$  for $ i=1,\ldots,k$.
		The generalized residual of
		${T}^{(n)}_{k}$ is represented by ${R}({T}^{(n)}_{k})$  defined as follows:
		\begin{align}\label{formresisue1}
		\nonumber {R}({T}^{(n)}_{k})&=\widetilde{T}_{k}( S_{n} )- {T}_{k}( {S}_{n} )\\
		&=\Delta {S}_{n} +\Delta {\mathscr{G}}_{k,n}\star \alpha^{}_k
		\end{align}

		\noindent For an arbitrary given set of tensors  ${Y}_{1}^{(n)}, \ldots,{Y}_{k}^{(n)}\in \mathbb{R}^{n_{1}\times n_2 \times n_{3} }$, let $\widetilde{\mathbf{H}}_{k,n} $ and  $\widetilde{\mathbf{L}}_{k,n}$  denote the subspaces     generated by $ \Delta {G}_{1}(n),\ldots,\Delta {G}_{k}(n)$  and   ${Y}_{1}^{(n)}, \ldots,{Y}_{k}^{(n)}$ respectively. Evidently, we have
		\begin{align}\label{residuerelationvector}
		{R}({T}^{(n)}_{k})-\Delta {S}_{n}\in \widetilde{\mathbf{H}}_{k,n}
		\end{align}
		and the unknown tensors $\alpha_i^{(n)}$ are determined by imposing the following condition,
		\begin{align*}
		{R}({T}^{(n)}_{k})\in \widetilde{\mathbf{L}}_{k,n}^{\perp}.
		\end{align*}
		
		\noindent Let $\mathscr{L}_{k,n}$ be a 4-mode tensor with frontal slices ${Y}_{1}^{(n)}, \ldots,{Y}_{k}^{(n)}$,
		the above relation  can be equivalently expressed by	
		\begin{align}
		\mathscr{L}_{k,n}&\diamondsuit {R}({T}^{(n)}_{k})=\mathscr{O},
		\end{align}
		where $\mathscr{O}$ is a zero tensor of order $1\times 1 \times n_3 \times k$.
		Therefore, from Proposition \ref{diamantpropos}, we have
		\begin{eqnarray*}
			\mathscr{O} & = &\mathscr{L}_{k,n} \diamondsuit {R}({T}^{(n)}_{k})\\
			& = & \mathscr{L}_{k,n}\diamondsuit (\Delta {S}_{n} +\Delta {\mathscr{G}}_{k,n}\star \alpha^{}_k)\\
			& = & \mathscr{L}_{k,n} \diamondsuit \Delta {S}_{n} + \mathscr{L}_{k,n} \diamondsuit  (\Delta {\mathscr{G}}_{k,n}\star \alpha^{}_k)= \mathscr{L}_{k,n} \diamondsuit \Delta {S}_{n} + (\mathscr{L}_{k,n} \diamondsuit \Delta {\mathscr{G}}_{k,n})\star \alpha^{}_k.
		\end{eqnarray*}
		
		\noindent In fact the unknown tensor $\alpha_k$ can be seen as the solution of thee following tensor equation,
		\[
		(\mathscr{L}_{k,n} \diamondsuit \Delta {\mathscr{G}}_{k,n})\star \alpha^{}_k =-\mathscr{L}_{k,n} \diamondsuit \Delta {S}_{n}.
		\]
	
		%
		%
		\noindent The choices of sequence of tensors $ {G}_{1}(n),\ldots,{G}_{k}(n)$  and   ${Y}_{1}^{(n)},\ldots,{Y}_{k}^{(n)}$ determine the type of  the  tensor  polynomial extrapolation method. In  fact, for all these polynomial-type  methods, the auxiliary sequence of  tensors  is given by ${G}_{i}(n)=\Delta {S}_{n+i-1}$ for $ i=1,\ldots,k$ ($n\geq0$). The following choices for ${Y}_{1}^{(n)},\ldots, {Y}_{k}^{(n)}$ can be used,
		\begin{align*}
		{Y}_{i}^{(n)}&= \Delta {S}_{n+i-1} \hspace{0.65cm}\text{for TMPE},\\
		{Y}_{i}^{(n)}&= \Delta^{2} {S}_{n+i-1}  \hspace{0.5cm}\text{for TRRE},\\
		{Y}_{i}^{(n)}&=  {Y}_{i}\hspace{1.6cm}  \text{for TMMPE},
		\end{align*}
		where the operator $\Delta^{2}$ refers to the second forward difference with respect to the index $n$ such that $$\Delta^{2} {S}_{n}=  \Delta {S}_{n+1}-  \Delta {S}_{n}\quad \text{and}
		\quad\Delta^{2} {G}_{i}(n)=\Delta {G}_{i}(n+1)-\Delta {G}_{i}(n),$$
		for $i=1,\ldots,k$.
		The approximation ${T}_{k}^{(n)}$  produced by TMPE, TRRE and TMMPE  can be also expressed  as follows:
		\begin{align*}
		{T}_{k}^{(n)}= \sum_{j=0}^{k} {S}_{n+j} \ast \gamma_{j}^{(k)}
		\end{align*}
		and the unknown tensors $\gamma_{0}^{(k)},\gamma_{1}^{(k)},\ldots,\gamma_{k}^{(k)}$ are determined by imposing the following condition
		\begin{align}\label{resid71}
		\sum_{j=0}^{k}   \gamma_{j}^{(k)}={\mathscr{I}_{n_2n_2n_3}}   \quad\text{and }\quad\sum_{j=0}^{k}\eta^{(n)}_{i,j}\ast\gamma_{j}^{(k)}=\mathscr{O}_{n_2n_2n_{3}}\quad 0\leq i<k
		\end{align}
		where  $\mathscr{O}_{n_2n_2n_{3}}\in {\mathbb R}^{n_2\times n_2 \times n_{3}} $ is the tensor which all  entries equal to zero,	 $\eta^{(n)}_{i,j}= ({Y}_{i+1}^{(n)})^T \ast \Delta {S} _{n+j}$ .\\
		
		\noindent From now on and for simplification, we set $n=0$. The system of equations (\ref{resid71}) is given in the following form 
		\begin{align}\label{sysofequa2}
		\begin{cases}
		\gamma_{0}^{(k)} +\gamma_{1}^{(k)} +\cdots+\gamma_{k}^{(k)} =\mathscr{I}_{n_2n_2n_3}\\ \\
	 {({Y}_{1}^{(0)})^T\ast \Delta {S} _{0}} \ast\gamma_{0}^{(k)}+   {({Y}_{1}^{(0)})^T\ast \Delta {S} _{1}}  \ast\gamma_{1}^{(k)}+\cdots+  {({Y}_{1}^{(0)})^T\ast \Delta {S} _{k}}  \ast\gamma_{k}^{(k)}=\mathscr{O}_{n_2n_2n_{3}}\\\\
	 {({Y}_{2}^{(0)})^T \ast \Delta {S} _{0}}  \ast\gamma_{0}^{(k)}+  {({Y}_{2}^{(0)})^T\ast \Delta {S} _{1}}  \ast\gamma_{1}^{(k)}+\cdots+  {({Y}_{2}^{(0)})^T \ast \Delta {S} _{k}} \ast\gamma_{k}^{(k)}=\mathscr{O}_{n_2n_2n_{3}}\\
		\ldots\ldots\ldots\ldots\ldots\ldots\\
		({Y}_{k}^{(0)})^T \ast \Delta {S} _{0}\ast\gamma_{0}^{(k)}+  {({Y}_{k}^{(0)})^T \ast \Delta {S} _{1}}  \ast\gamma_{1}^{(k)}+\cdots+   {({Y}_{k}^{(0)})^T \ast \Delta {S} _{k}}  \ast\gamma_{k}^{k}=\mathscr{O}_{n_2n_2n_{3}}\\
		\end{cases}.	
		\end{align}
		
		\noindent Let $\beta_{i}^{(k)} =  \gamma_{i}^{(k)} \ast (\gamma_{k}^{(k)})^{-1} $  where $(\gamma_{k}^{(k)})^{-1} $ is the inverse of  $\gamma_{k}^{(k)}$, i.e.,   $\gamma_{k}^{(k)}\ast (\gamma_{k}^{(k)})^{-1}=\mathscr{I}_{n_2n_2n_3}$,  for $0\leq l\leq k$. Then, it is not difficult to verify that
		\begin{align}\label{37}
		\gamma_{i}^{(k)} = \beta_{i}^{(k)} \ast ({\sum_{i=0}^{k} \beta_{i}^{(k)} })^{-1}  \quad \text{for}\quad 0\leq l<k \quad \text{and} \quad  \beta_{k}^{(k)} =\mathscr{I}_{n_2n_2n_3}.
		\end{align} 	
		The system of equations  (\ref*{sysofequa2}) can be rewritten as follows 
		\begin{align}\label{sysofequa3}
		\begin{cases}
		{({Y}_{1}^{(0)})^T\ast \Delta {S} _{0}}  \ast\beta_{0}^{(k)}+   \cdots+  {({Y}_{1}^{(0)})^T \ast \Delta {S} _{k-1}}  \ast\beta_{k-1}^{(k)}=- {({Y}_{1}^{(0)})^T \ast \Delta  S _{k}}   \\
		\ldots\ldots\ldots\ldots\ldots\ldots\ldots\\
		 {({Y}_{k}^{(0)})^T\ast \Delta {S} _{0}} \ast\beta_{0}^{(k)} +\cdots+  {({Y}_{k}^{(0)})^T \ast \Delta {S} _{k-1}}  \ast\beta_{k-1}^{(k)}=- {({Y}_{k}^{(0)})^T \ast \Delta {S} _{k}}
		\end{cases}.	
		\end{align}
		The above system can be mentioned in the following form
		\begin{align}\label{betaq0}
		(\mathscr{L}_{k,n}\diamondsuit   \mathscr{V} _{k})\star {\bf \beta}_k =-(\mathscr{L}_{k,n}\diamondsuit \Delta{{S}_{k}})
		\end{align}
		where ${\bf  \beta}_k$ is  the  4-mode tensor with the $k$ frontal slices ${ \beta}^{(k)}_0,\ldots,{ \beta}^{(k)}_{k-1}$ and  $\mathscr{V} _{k}$ is a 4-mode tensor whose $i$-th frontal slice is given by $\Delta {S} _{i-1}$ for $i=1,2,\ldots,k$. \\
		Having $\gamma_{0} ,\gamma_{1} ,\ldots,\gamma_{k} $ computed, we set 
		\begin{align}\label{scalairesalpha}
		\alpha_{0}^{(k)} =\mathscr{I}_{n_2n_2n_3}-\gamma_{0}^{(k)} ,\quad \alpha_{j}^{(k)} =\alpha_{j-1}^{(k)} -\gamma_{j}^{(k)} ,\quad 1\leq j <k \quad \text{and}\quad \alpha_{k-1}^{(k)} =\gamma_{k}^{(k)}.
		\end{align}
		Setting  ${T}_{k}= {T}_{k}^{(0)}$, we get
		\begin{align}\label{resid13}
		{T}_{k} ={S}_{0}+ \sum_{j=0}^{k-1} {V}_{j} \ast\alpha_{j}^{(k)}   ={S}_{0}+  \mathscr{V}_{k} \star {\alpha }_k,
		\end{align}
		where $ {V}_{j}=\Delta {S} _{j}$ the $(j+1)$-th frontal slice of $\mathscr{V}_{k}$ for $j=0\ldots,k-1$ and
		${\alpha}_k$ is a 4-mode tensor with frontal slices ${ \alpha}^{(k)}_0,\ldots,{\alpha}^{(k)}_{k-1}$.
		%
		To determine  $\gamma^{(k)}_i$ for $i=0,1,\ldots,k$, we first we need to compute  $\beta^{(k)}$ by solving system of equations (\ref{betaq0}). Using (\ref{formresisue1}), \eqref{scalairesalpha} and  (\ref{resid13}),  the generalized residual ${R}({T} _{k})$ can be also seen as follows:
		\begin{align}\label{betaq}
		{R}(T _{k})=\sum_{i=0}^{k} {V}_{i}\ast \gamma_{i}^{(k)}=\mathscr{V}_{k}\star \gamma_k
		\end{align}
		in which ${\gamma}_k$ and $\mathscr{V}_{k}$ are 4-mode tensors with whose $i$-th frontal slices are respectively given by ${\gamma}^{(k)}_{i-1}$ and $\bar{V}_{i-1}$ for $i=1,2,\ldots,k$.

		\subsection{The tensor toplogical $\epsilon$- transformation}\label{sub3.2}
		For vector sequences, Brezinski \cite{B0032}  proposed the well known  topological $\epsilon$-algorithm (TEA) which is a generalization of the scalar $\epsilon$-algorithm \cite{Wynn} known as a technique for transforming slowly convergent or divergent sequences.  In this section, we briefly see how to extend  this idea  in tensor framework and define the Tensor Toplogical $\epsilon$-Transformation (TTET).\\
		Let  $({S}_{n})$ a given  sequence of tensors in $\mathbb{R}^{n_{1}\times n_2\times n_{3} }$,   and
		consider approximations of tensors ${E}_{k}({S}_{n})={E}^{(n)}_{k}$ of the limit of sequence $({S}_{n})_{n\in \mathbb{N}} $  defined as
		\begin{align}\label{TTET}
		{E}_{k}^{(n)} = { S}_{n}+ \sum_{i=1}^{k}  \Delta {S}_{n+i-1}\ast\beta^{(n)}_{i},\quad  n\geq 0.
		\end{align}
		where $\beta^{(n)}_{i}\in \mathbb{R}^{n_2\times n_2 \times n_{3} }$ are  the unknown tensors  to be determined for $i=1\ldots,k$. We set
		\begin{align*}
		\widetilde{{E}}_{k,j}^{(n)} = {  S}_{n+j}+ \sum_{i=1}^{k}  \Delta {S}_{n+i+j-1}\ast\beta^{(n)}_{i}   \quad j=1,\ldots,k,
		\end{align*}
		where $\widetilde{ {E}}_{k,0}^{(n)}={E}^{(n)}_{k}$. Let $ \widetilde{R}_{j}(\bar{E}^{(n)}_{k})$
		denote the $j$-th generalized residual tensor, i.e.,
		\begin{align*}
		{R}_{j}({E}^{(n)}_{k})=\widetilde{{E}}_{k,j}^{(n)}-\widetilde{{E}}_{k,j-1}^{(n)}.
		\end{align*}
		For given third order tensor ${Y}\in   \mathbb{R}^{n_{1}\times n_2\times n_{3} }$ be given. The coefficients  $\beta^{(n)}_{i}$  in (\ref{TTET}) are computed  such that
		\begin{align*}
		\sum_{i=1}^{k}  ({Y}^{T} \ast  \Delta^2 {S}_{n+i+j-1}) \ast\beta^{(n)}_{i}=\mathscr{O}, \qquad j=0,1\ldots,k-1.
		\end{align*}
		where $\mathscr{O}\in \mathbb{R}^{n_2\times n_2\times n_{3}}$. The above conditions result the following system of equations:
		\begin{align*}
		\begin{cases}
		({Y}^{T}\ast\Delta^{2} {S}_{n}) \ast\beta_{1}^{(n)}+\cdots+ ({Y}^{T}\ast 	\Delta^{2}  {S}_{n+k-1}) \ast\beta_{k}^{(n)} =-{Y}^{T}\ast	\Delta {S}_{n}  \\
		\qquad \vdots\\
		({Y}^{T}\ast\Delta^{2}  {S}_{n+k-1})  \ast\beta_{1}^{(n)} +\cdots+ ({Y}^{T}\ast	\Delta^{2}  {S}_{n+2k-2})\ast\beta_{k}^{(n)}=-{Y}^{T}\ast	\Delta {S}_{n+k-1}
		\end{cases}.	
		\end{align*}
		In the case that the above system is uniquely solvable we can  obtain ${E}^{(n)}_{k}$.

	 \section{Application of tensor extrapolation methods to solve ill-posed tensor problems }\label{TRRETTSVD0}
	
The purpose of this section is to adopt the idea used by  Jbilou et al. \cite{B366} for solving a class of ill-posed tensor equations. To this end, we start by recalling the truncated tensor SVD for the third order-tensors (TTSVD), then we prove a theorem
which gives the minimum norm (least-squares) solution of our mentioned tensor equation. In addition, we present an algorithm for approximating the Moore-Penros inverse of tensor. In the second subsection, we combine TTSVD with the TSRRE method to resolve tensor ill-posed problems.
	
	 \subsection{Truncated tensor singular value decomposition}\label{TTSVD0}
	
The truncated SVD (T-SVD) of matrices is efficient for approximating Moore--Penrose inverse and solving least-squares problem as the truncated version consumes less space of storage  in comparison with SVD when the rank of matrix is not very large. This inspired Miao et al. \cite{B30112} to extend the theory of tensor SVD (TSVD) \cite{B0290} to truncated tensor  SVD (TTSVD); here the decompositions are  based on the T-product of tensors. In this part, first, the TSVD and TTSVD are recalled.  Then we derive the explicit form for minimum norm solution of least-square problem in tensor framework.
	
	 In the following, an  F-diagonal tensor refers to a third order tensor whose all frontal slices are diagonal. \\
	 \begin{theorem}\cite{B0290} \label{th4.1}
	 	Let $\mathscr{A}\in\mathbb{R}^{n_1\times n_2\times n_3}$ be a real valued-tensor, then  	there exists  orthogonal tensors  $\mathscr{U}\in \mathbb{R}^{n_1\times n_1\times n_3} $ and $\mathscr{V}\in \mathbb{R}^{n_2\times n_2\times n_3} $ such that
	 	\begin{align}\label{TCSVD}
	 	\mathscr{A}=\mathscr{U}\ast\mathscr{S}\ast \mathscr{V}^{T}
	 	\end{align}
	 	in which $\mathscr{S}\in\mathbb{R}^{n_1\times n_2\times n_3} $ is an F-diagonal tensor .
	 \end{theorem}
	
	 \medskip
	 \begin{theorem}\cite{B30112}
	 	Let $\mathscr{A}\in\mathbb{R}^{n_1\times n_2 \times n_3}$ be a real valued-tensor,  then 	there exist unitary tensors  $\mathscr{U}_{(k)}\in \mathbb{R}^{n_1\times k\times n_3} $ and $\mathscr{V}_{(k)}\in \mathbb{R}^{n_2\times k\times n_3} $ such that
	 	\begin{align}\label{TCSVD1}
	 	\mathscr{A}\approx\mathscr{A}_{k}=\mathscr{U}_{(k)}\ast\mathscr{S}_{(k)}\ast \mathscr{V}_{(k)}^{T}
	 	\end{align}
	 	and the Moore--Penrose inverse of tensor $\mathscr{A}_k$ is given by
	 	\begin{align*}
	 	\mathscr{A}_{k}^\dagger=\mathscr{V}_{(k)}\ast\mathscr{S}_{(k)}^\dagger \ast \mathscr{U}_{(k)}^{T},
	 	\end{align*}
	 	where $\mathscr{S}_{k}\in \mathbb{R}^{k\times k\times n_3} $ is an F-diagonal tensor and $k<min(n,m)$ is called the \textit{tubal-rank} of $\mathscr{A}$ based on the T-product.
	 \end{theorem}
	
	 \medskip
	 \noindent In \cite{B30112}, the decomposition (\ref{TCSVD1}) is called the tensor compact SVD (T-CSVD) of $\mathscr{A}$ and the tensor $\mathscr{A}_{k}$ can be regarded  the rank-k approximation of the tensor $\mathscr{A}$. Here we call it TTSVD. 
	 We comment here that a  similar compression strategy to \eqref{TCSVD1} is also
	 given by Kilmer and Martin\cite{B0290}. \\
	 The TTSVD  of  $\mathscr{A}_{k}$ can be expressed as follows
	 \begin{align}\label{Eq26}
	 \mathscr{A}_{k}=\sum_{j=1}^{k} \bar{U}_{ j}  \ast d_{j}\ast \bar{V}_{j}^{T}
	 \end{align}
	 where $\bar{U}_j=\mathscr{U}_{(k)}(:,j,:) \in \mathbb{R}^{n_1\times 1\times n_3} $, $\bar{V}_j=\mathscr{V}_{(k)}(:,j,:)\in \mathbb{R}^{n_2\times 1\times n_3}$ and $d_{j}=\mathscr{S}(j,j,:)\in \mathbb{R}^{1\times 1\times n_3} $  for $j=1,2,\ldots,k$.\\
	
	 \begin{remark}\label{rem4.1}
	 	Let $\mathscr{A}=\mathscr{U}\ast \mathscr{S} \ast \mathscr{V}^T$ be the TSVD of $\mathscr{A}$. In view of Proposition \ref{prop2.1}, we can see that
	 	\[
	 	\bar{U}_i \ast \bar{U}_j=(\mathscr{U}^T \ast \mathscr{U})_{ij:} \qquad
	 	\rm{and} \qquad 	\bar{V}_i \ast \bar{V}_j=(\mathscr{V}^T \ast \mathscr{V})_{ij:}.
	 	\]
	 	Therefore, we have $\bar{U}_i \ast \bar{U}_j$ and  $\bar{V}_i \ast \bar{V}_j$ are zero tubal-scalar of length $n_3$ for $i\ne j$;
	 	$\bar{U}_i \ast \bar{U}_i={\rm \bf e}$ and
	 	$\bar{V}_i \ast \bar{V}_i={\rm \bf e}.$
	 \end{remark}
	
	 \medskip
	 \noindent The following theorem reveals that  $\mathscr{A}_{k}$ is an optimal approximation of a tensor, see \cite{B0290} for the proof.
	
	 \begin{theorem}\label{th4.3}
	 	Let the TSVD of $\mathscr{A}\in \mathbb{R}^{n_1\times n_2 \times n_3}$ be given by
	 	$\mathscr{A}=\mathscr{U} \ast \mathscr{S} \ast \mathscr{V}^T$ and for
	 	$k < \min(n_1,n_2)$, the tensor $\mathscr{A}_{k}$ given by \eqref{Eq26}
	 	satisfies
	 	\[
	 	\mathscr{A}_k={\rm argmin}_{\tilde{\mathscr{A}} \in M} \|\mathscr{A} - \tilde{\mathscr{A}}\|
	 	\]
	 	where $M=\{\mathscr{C}=\mathscr{X}\ast\mathscr{Y}~|~\mathscr{X}\in \mathbb{R}^{n_1\times k\times n_3}, \mathscr{Y}\in \mathbb{R}^{k\times n_2\times n_3}\}$.
	 \end{theorem}
	
	 \medskip
	 \noindent In view of Theorem \ref{th4.3}, the Moore-Penrose inverse of tensor  $\mathscr{A}$ is efficiently estimated on $\widetilde{M}=\{\mathscr{C}=\mathscr{X}\ast\mathscr{Y}~|~\mathscr{X}\in \mathbb{R}^{n_2\times k\times n_3}, \mathscr{Y}\in \mathbb{R}^{k\times n_1\times n_3}\}$ by
\begin{align}\label{Aapprox}
	 \mathscr{A}^\dagger \approx \sum_{j=1}^{k} \bar{V}_{ j}  \ast d^{\dagger}_{j}\ast \bar{U}_{j}^{T},
\end{align}
in which tubal-scalar $d^{\dagger}_{j}\in \mathbb{R}^{1\times 1 \times n_3}$ stands for the  $(j,j,:)$ entry of   $\mathscr{S}_{(k)}^{\dagger}$. \\
	
	 \noindent The following theorem has a key role in deriving the results of the next section.
	 \begin{theorem}\label{th4.4} Assume that $\mathscr{A}\in \mathbb{R}^{n_1\times n_2 \times n_3}$ and $\mathscr{B}\in \mathbb{R}^{n_1\times s\times n_3}$. If $\hat{\mathscr{X}}=\mathscr{A}^\dagger \ast \mathscr{B}$, then
	 	\begin{equation}
	 	\|\mathscr{A}\ast \hat{\mathscr{X}} -\mathscr{B}\|=\min_{\mathscr{X}\in \mathbb{R}^{n_2\times s\times n_3}} \|\mathscr{A}\ast \mathscr{X} -\mathscr{B}\|.
	 	\end{equation}
	 	For any $\tilde{\mathscr{X}} \in \mathbb{R}^{n_1\times s\times n_3}$ such that $\tilde{\mathscr{X}}\ne \hat{\mathscr{X}} $ and $\|\mathscr{A}\ast \hat{\mathscr{X}} -\mathscr{B}\|=\|\mathscr{A}\ast \tilde{\mathscr{X}} -\mathscr{B}\|$, then $\|\hat{\mathscr{X}}\| < \|\tilde{\mathscr{X}}\|$.
	 \end{theorem}
	 \begin{proof}  
	 	Let $A=\mathscr{U}^T\ast \mathscr{S} \ast \mathscr{V}$ be the TSVD of $\mathscr{A}$. From \cite[Lemma 3.19]{B0290}, it can be seen that
	 	\[
	 	\|\mathscr{A}\ast \mathscr{X} -\mathscr{B}\|=\|\mathscr{U}^T \ast (\mathscr{A}\ast \mathscr{X} -\mathscr{B})\|.
	 	\]
	 	By some straightforward computations, we have
	 	\begin{eqnarray*}
	 		\|\mathscr{U}^T \ast (\mathscr{A}\ast \mathscr{X} -\mathscr{B})\| & = & \|  \mathscr{S} \ast \mathscr{V}^T \ast \mathscr{X} -\mathscr{U}^T \ast \mathscr{B}\|.
	 	\end{eqnarray*}
	 	Setting $\mathscr{Z}=\mathscr{V}^T \ast \mathscr{X}$ and $\mathscr{W}=\mathscr{U}^T \ast \mathscr{B}$, we get
	 	\begin{eqnarray}
	 		\nonumber \|\mathscr{U}^T \ast (\mathscr{A}\ast \mathscr{X} -\mathscr{B})\| & = &\|\mathscr{S} \ast \mathscr{Z} - \mathscr{W}\|\\
	 		 \nonumber &=&\|(F_{n_3}\otimes I_{n_1}) {\rm MatVec} (\mathscr{S} \ast \mathscr{Z})- (F_{n_3}\otimes I_{n_1}){\rm MatVec} (\mathscr{W})\|_F\\
	 		&=&\|(F_{n_3}\otimes I_{n_1}){\rm bcirc}(\mathscr{S})  {\rm MatVec} (\mathscr{Z})- (F_{n_3}\otimes I_{n_1}){\rm MatVec} (\mathscr{W})\|_F \label{eqnn}
	 	\end{eqnarray}
	where  $\|\cdot\|_F$ is the well-known Frobenius matrix norm, the notation $\otimes$ stands for the Kronecker product and the matrix $F_{n_3}$ is  the discrete Fourier matrix of size $n_3\times n_3$ defined by
	 	(see \cite{Chan})
	 	\[{F_{{n_3}}} = \frac{1}{{\sqrt {{n_3}} }}\left( {\begin{array}{*{20}{c}}
	 		1&1&1&1& \cdots &1\\
	 		1&{{\omega ^{}}}&{{\omega ^2}}&{{\omega ^3}}& \cdots &{{\omega ^{{n_3} - 1}}}\\
	 		1&{{\omega ^2}}&{{\omega ^4}}&{{\omega ^6}}& \cdots &{{\omega ^{2({n_3} - 1)}}}\\
	 		1&{{\omega ^3}}&{{\omega ^6}}&{{\omega ^9}}& \cdots &{{\omega ^{3({n_3} - 1)}}}\\
	 		\vdots & \vdots & \vdots & \vdots & \ddots & \vdots \\
	 		1&{{\omega ^{{n_3} - 1}}}&{{\omega ^{2({n_3} - 1)}}}&{{\omega ^{3({n_3} - 1)}}}& \cdots &{{\omega ^{({n_3} - 1)({n_3} - 1)}}}
	 		\end{array}} \right),\]
	 	\noindent in which $\omega =e^{-2\pi {\bf i} /n_3}$ is the primitive $n_3$--th root of unity in which ${\bf i}=\sqrt{-1}$ and $F^*$
	 	denotes the conjugate transpose of $F$.	 	

For simplicity we set $Z=(F_{n_3}\otimes I_{n_2}){\rm MatVec} (\mathscr{Z})$ and $W=(F_{n_3}\otimes I_{n_1}){\rm MatVec} (\mathscr{W})$, hence Eq.\ref{eqnn}
can be rewritten as follows: 
	\begin{eqnarray*}
	 		\|\mathscr{U}^T \ast (\mathscr{A}\ast \mathscr{X} -\mathscr{B})\| 
	 		&=&\|(F_{n_3}\otimes I_{n_1}){\rm bcirc}(\mathscr{S}) (F_{n_3}^*\otimes I_{n_2}) Z- W\|_F.
	 	\end{eqnarray*}
Notice that
	 	\begin{equation} \label{Eq30}
	 	\|Z\|_F=\|{\rm MatVec} (\mathscr{Z})\|_F=\|\mathscr{Z}\|=\|\mathscr{V}^T\ast \mathscr{X}\|=\|\mathscr{X}\|.
	 	\end{equation}	
	 	From \cite{B0290,B30112}, it is known that there exists diagonal (rectangular) matrices $\Sigma_1,\Sigma_2,\ldots,\Sigma_{n_3}$ such that
	 	\[
	 	(F_{n_3}\otimes I_{n_1}){\rm bcirc}(\mathscr{S}) (F_{n_3}^*\otimes I_{n_2}) =\Sigma=\left( {\begin{array}{*{20}{c}}
	 		{{\Sigma_1}}&{}&{}&{}\\
	 		{}&{{\Sigma_2}}&{}&{}\\
	 		{}&{}& \ddots &{}\\
	 		{}&{}&{}&{{\Sigma_{{n_3}}}}
	 		\end{array}} \right),
	 	\]
	 	 and
	 	\[
	 	(F_{n_3}^*\otimes I_{n_1}){\rm bcirc}(\mathscr{S^\dagger}) (F_{n_3}\otimes I_{n_2}) =\Sigma^\dagger=\left( {\begin{array}{*{20}{c}}
	 		{{\Sigma_1^\dagger}}&{}&{}&{}\\
	 		{}&{{\Sigma_2^\dagger}}&{}&{}\\
	 		{}&{}& \ddots &{}\\
	 		{}&{}&{}&{{\Sigma_{{n_3}}^\dagger}}
	 		\end{array}} \right).
	 	\]
From the above computations, we have
	 	\[
	 	\|\mathscr{A}\ast \mathscr{X} -\mathscr{B}\| = \|\Sigma~ Z- W\|_F.
	 	\]
	 	It is well-known from the literature that the minimum Frobenius norm solution of
	 	$\|\Sigma~ Z- W\|_F$
	 	over $\mathbb{R}^{n_2n_3\times s}$ is given by $\hat{Z}=\Sigma^\dagger W$, i.e.,
	 	$$\hat{Z}={\rm argmin}_{Z\in \mathbb{R}^{n_2n_3\times s}} \|\Sigma~ Z- W\|_F.$$
	 	Let $\hat{\mathscr{X}}$ be the tensor such that
	 	\begin{eqnarray*}
	 		{\rm MatVec}(\mathscr{V}^T\ast \hat{\mathscr{X}})
	 		& = &  (F_{n_3}^*\otimes I_{n_2}) \Sigma^\dagger W \\
	 		& = & (F_{n_3}^*\otimes I_{n_2}) \Sigma^\dagger (F_{n_3}\otimes I_{n_1}){\rm MatVec} (\mathscr{W})\\
	 		& = & {\rm bcirc}(\mathscr{S}^\dagger) {\rm MatVec}(\mathscr{U}^T\ast \mathscr{B}).
	 	\end{eqnarray*}
	 	It is immediate to deduce the following equality
	 	\[
	 	{\rm MatVec}(\mathscr{V}^T\ast \hat{\mathscr{X}}) = {\rm MatVec}(\mathscr{S}^\dagger\ast \mathscr{U}^T\ast \mathscr{B}),
	 	\]
	 	which is equivalent to say that $\mathscr{V}^T\ast \hat{\mathscr{X}}=\mathscr{S}^\dagger\ast \mathscr{U}^T\ast \mathscr{B}$.
	 	Finally, the result follows from  \eqref{Eq30}.
	 \end{proof}
	
	 \medskip
	
	 \noindent Similar to the TSVD \cite{B0290}, the TTSVD can be obtained
	 using the fast Fourier transform. Notice that circulant matrices are diagonalizable via the normalized DFT.
	 This fact was used for proving Theorem \ref{th4.1} and deriving the Matlab pseudocode for TSVD; see
	 Algorithm \ref{TSVD} for more details. Here,  we further use this fact and present a Matlab pseudocode
	 for computing TTSVD and the corresponding approximation of Moore-Penrose inverse in Algorithm
	 \ref{TCSVDalg}. Note that the Matlab function \verb|pinv(.)| reduces to \verb|inv(.)|
	 when the diagonal matrix $S$ in Step 2 of the algorithm is nonsingular.


	
	
	 \begin{algorithm}[h]
	 	\caption{Matlab pseudocode for TSVD \cite{B029}}\label{TSVD}
	 	{\bf Input.} $\mathscr{A}\in \mathbb{R}^{n_1\times n_2 \times n_3}$;\\
	 	Step 1. Set
	 	$\mathscr{D}=\rm{fft}(\mathscr{A},[~],3)$;\\
	 	Step 2. for $i=1,2,\ldots,n_3$\\
	 	$~~~~~~~~~~~~~~~~~~~[U,S,V]={\rm{svd}}(\mathscr{D}(:,:,i))$;\\
	 	$~~~~~~~~~~~~~~~~~~~\mathscr{U}(:,:,i)=U$;
	 	$\mathscr{V}(:,:,i)=V$;
	 	$\mathscr{S}(:,:,i)=S$;\\
	 	$~~~~~~~~~~~~$end\\
	 	Step 3. $\mathscr{U}_{}={\rm{ifft}}(\mathscr{U},[~],3)$; $\mathscr{V}_{}={\rm{ifft}}(\mathscr{V},[~],3)$; $\mathscr{S}_{}={\rm{ifft}}(\mathscr{S},[~],3)$;\\
	 	{\bf Output.} $\mathscr{A}=\mathscr{U}_{}\ast \mathscr{S}_{} \ast \mathscr{V}^T_{}$.
	 \end{algorithm}

	
	 \begin{algorithm}[h]
	 	\caption{Matlab pseudocode for TTSVD and the corresponding approximation of  Moore--Penrose inverse}\label{TCSVDalg}
	 	{\bf Input.} $\mathscr{A}\in \mathbb{R}^{n_1\times n_2 \times n_3}$ and $k\le \min(n_1,n_2)$;\\
	 	Step 1. Set
	 	$\mathscr{D}=\rm{fft}(\mathscr{A},[~],3)$;\\
	 	Step 2. for $i=1,2,\ldots,n_3$\\
	 	$~~~~~~~~~~~~~~~~~~~[U,S,V]={\rm{svd}}(\mathscr{D}(:,:,i))$;\\
	 	$~~~~~~~~~~~~~~~~~~~\mathscr{U}(:,:,i)=U(:,1:k)$;
	 	$\mathscr{V}(:,:,i)=V(:,1:k)$;\\
	 	$~~~~~~~~~~~~~~~~~~~\mathscr{S}(:,:,i)=S(1:k,1:k)$;
	 	$\mathscr{S}^\dagger(:,:,i)={\rm{pinv}}(S(1:k,1:k))$;\\
	 	$~~~~~~~~~~~~$end\\
	 	Step 3. $\mathscr{U}_{(k)}={\rm{ifft}}(\mathscr{U},[~],3)$; $\mathscr{V}_{(k)}={\rm{ifft}}(\mathscr{V},[~],3)$; $\mathscr{S}_{(k)}={\rm{ifft}}(\mathscr{S},[~],3)$; $\mathscr{S}_{(k)}^\dagger={\rm{ifft}}(\mathscr{S}^\dagger,[~],3)$;\\
	 	{\bf Output.} $\mathscr{A}_k=\mathscr{U}_{(k)}\ast \mathscr{S}_{(k)} \ast \mathscr{V}^T_{(k)}$ and
	 	$\mathscr{A}_k^\dagger=\mathscr{V}_{(k)}\ast \mathscr{S}_{(k)}^\dagger \ast \mathscr{U}^T_{(k)}$.
	 \end{algorithm}

	 	 \subsection{Tensor extrapolation methods applied to TTSVD   sequences in  ill-posed problems}\label{sub4.2}
	 	
	 	 Let $\mathscr{A}\in \mathbb{R}^{n_{1}\times n_{2}\times n_{3} }$ be given and consider its TSVD. That is let us apply Algorithm \ref{TSVD} for $\mathscr{A}$. In the case that the matrix ${\rm bcirc}(\mathscr{A})$ has too many singular values being close to zero, the tensor  $\mathscr{A}$ is called ill-determined rank.\\
	 	
	 	 \noindent In this subsection, we discuss solutions of  tensor equations in the following form
	 	 \begin{align}\label{equ1}
	 	 \mathscr{A}\ast \widetilde{\mathscr{X}}=\widetilde {\mathscr{B}}
	 	 \end{align}
	 	 where the ill-determined rank tensor $\mathscr{A}\in \mathbb{R}^{n_{1}\times n_{2}\times n_{3} }$  and right-hand side $\widetilde {\mathscr{B}}\in \mathbb{R}^{n_{1}\times n_{2}\times n_{3} }$ are given and $\mathscr{X}\in \mathbb{R}^{n_{2}\times n_{2}\times n_{3} }$ is an unknown tensor to be determined when the values of $n_2$ and $n_3$ are small or moderate.
	 	 Our  goal here is to find an approximate solution for the above tensor equation.
	 	 %
	 	 %
	 	 %
	 	 Systems of tensors equations (\ref{equ1})  with a tensor of ill-determined rank often are referred to as linear discrete tensor ill-posed problems. They arise   in several areas in science and engineering, such as the restoration of color
	 	 and multispectral images \cite{B13,B015,B016}, blind source separation \cite{B17},  when
	 	 one seeks to determine the cause of an observed effect. In these applications, the right-hand side $\widetilde {\mathscr{B}} \in \mathbb{R}^{n_{1}\times n_2\times n_{3} }$ is typically
	 	 contaminated by an error $\mathscr{E}$, i.e.,
	 	 $
	 	 \widetilde {\mathscr{B}}=\tilde{\bar {\mathscr{B}}}+\mathscr{E}
	 	 $
	 	 where $\tilde{\bar {\mathscr{B}}}$ denotes the unknown error-free right-hand side.
	 	 We are interested in determining an accurate approximation of the (least-squares) solution $\tilde{\bar {\mathscr{X}}}$ of the following (in)consistent  tensor equation
	 	 \begin{align*}
	 	 \mathscr{A}\ast \tilde{\bar {\mathscr{X}}}=\widetilde{\bar {\mathscr{B}}}
	 	 \end{align*}
	 	 with error-free right-hand side by solving \eqref{equ1}. From Theorem \ref{th4.4}, it is known that the exact solution is given by
	 	 \begin{align*}
	 	 \tilde{\bar {\mathscr{X}}}=\mathscr{A}^\dagger\ast  \tilde{\bar {\mathscr{B}}}.
	 	 \end{align*}
	 	
	 	 In the matrix case, when the size of the rank deficient matrix $A$ is moderate,  a popular method for computing an approximation for the (least-squares) solution of (in)consistent linear system of equations $Ax=b$ consist in using  the Truncated SVD which replaces the matrix $A^\dagger$ by a low-rank approximation; see, e.g., Golub and Van Loan \cite{B08} or Hansen \cite{B081}.  Following the same idea, we approximate   $\widetilde{\mathscr{X}}$ by  using the expression (\ref{Aapprox}).
	 	 In view of Theorem \ref{th4.3}, we approximate $\mathscr{A}_{}^\dagger\ast  \bar {\mathscr{B}}$ by $\widetilde {\mathscr{X}}_{k}$ given as follow:
	 	 \begin{align*}
	 	 \widetilde {\mathscr{X}}_{k}=\sum_{j=1}^{k} \bar{V}_{ j}  \ast d^{\dagger}_{j}\ast \bar{U}_{j}^{T} \ast   \widetilde  {\mathscr{B}}.
	 	 \end{align*}
	 	 For the matrix case, Jbilou  et al. \cite{B360}  proposed the application of the RRE to the sequence of vectors generated by the truncated SVD. In the remainder of this paper, we follow the same idea and consider the application of TRRE to the sequence of tensors generated by the truncated TSVD (TTSVD). To do so, we set
	 	 \begin{align*}
	 	 {Y}_{i} =\Delta ^{2} {S}_{i-1}    \quad  1\leq i\leq k-2,
	 	 \end{align*}
	 	 in which  $(\bar{S}_{k})_{k\geq 0}$ is the tensor sequence  generated by the  TTSVD. Thus,
	 	 \begin{align}\label{ttsvdtrre}
	 	 {S}_{k} &=\mathscr{A}_{k}^\dagger \ast \widetilde {\mathscr{B}}= \sum_{j=1}^{k} \bar{V}_{ j}  \ast  \delta_{j}
	 	 \end{align}
	 	 where  $\delta_{j}= d^{\dagger}_{j}\ast \bar{U}_{j}^{T} \ast  \bar{\mathscr{B}}$ and ${S}_{0}$ is set to be a zero tensor of order ${n_2\times n_2\times n_3}$.
	 	 It can be observed that
	 	 \begin{align}
	 	 \Delta {S}_{k-1}={S}_{k}-\bar{S}_{k-1}= \bar{V}_{k}\ast\delta _{k}.
	 	 \end{align}

	 	 We assume here that $\delta_ {k}^T\ast \delta_ {k}$ is invertible and notice that if $\delta_ {k}^T\ast \delta_ {k}$ is zero, then we can delete the corresponding member from the sequence (\ref{ttsvdtrre}) and compute the next one by keeping the same index notation. To overcome the cases $\delta_ {k}^T\ast \delta_ {k}$ is numerically non invertible,  we can use a small shift $\epsilon \mathscr{I}$ (say $\epsilon=1e-10$) on the positive semi definite tensor $\delta_ {k}^T\ast \delta_ {k}$; see Remark \ref{rem1.2new}. \\ Let  $\Delta{\mathscr{S}_{k}}$ and $\Delta^2{\mathscr{S}_{k}}$ be  4-mode tensors whose
	 	 $i$-th frontal slices are given by $\Delta {S}_{i-1}$ and  $\Delta^2 {S}_{i-1}$ for $i=1,2,\ldots, k$, respectively, for $i=1,2,\ldots, k$
	 	
	 	 %
	 	 \noindent Notice  that  $\Delta^2 {S}_{j-1}= \bar{ V}_{j+1}\ast\delta_{j+1} -\bar{ V}_{j}\ast\delta_{j}$ for $j=1,2,\ldots k$.  Hence,
	 	 \begin{align*}
	 	 (\Delta^2 {\mathscr{S}_{k}}  \diamondsuit \Delta {\mathscr{S}_{k}})_{:::ji}&=(\delta_{i+1}^T\ast  \bar{ V}_{i+1}^T-\delta_i^T\ast  \bar{ V}_{i}^T)\ast \bar{V}_{j}\ast\delta _{j} \\
	 	 &=\delta_{i+1}^T\ast  \bar{ V}_{i+1}^T\ast \bar{V}_{j}\ast\delta _{j}-\delta_i^T\ast  \bar{ V}_{i}^{T}\ast \bar{V}_{j}\ast\delta _{j}\quad {\rm for} \quad i,j=1,2,\ldots,k.
	 	 \end{align*}
	 	
	 	 \noindent Using  Remark \ref{rem4.1}, we can deduce that the nonzero frontal slices of $\Delta^2\tilde{\mathscr{S}_{k}}  \diamondsuit \Delta {\mathscr{S}_{k}}$ are given by
	 	 \[
	 	 (\Delta^2 {\mathscr{S}_{k}}  \diamondsuit \Delta {\mathscr{S}_{k}})_{:::(i+1)i}=\delta_{i+1}^T\ast \delta _{i+1},
	 	 \]
	 	 and
	 	 \[
	 	 (\Delta^2{\mathscr{S}_{k}}  \diamondsuit \Delta {\mathscr{S}_{k}})_{:::ii}=-\delta_{i}^T\ast \delta _{i}.
	 	 \]
	 	 Straightforward computations together with Remark \ref{rem4.1} show that the frontal slices of the 4-mode tensor
	 	 $(\Delta^2 {\mathscr{S}_{k}}   \diamondsuit \Delta {S}_{k})$ are equal to zero except the last frontal slice being equal to $\delta_{k+1}^T \ast \delta_{k+1}$.
	 	 For notational simplicity, we define
	 	 $\Theta_{i+1}=\delta_{i+1}^T\ast\delta_{i+1}$ for $i=0,\ldots k$.
	 	 In summary, we need to solve the following tensor equation
	 	 \begin{align*}
	 	 (\Delta^2 {\mathscr{S}_{k}}  \diamondsuit \Delta {\mathscr{S}_{k}}) \star {\beta}{_k} =-(\Delta^2 {\mathscr{S}_{k}}   \diamondsuit \Delta  {S}_{k}),
	 	 \end{align*}
	 	 where $ {\beta}{_k}$ is a 4-mode tensor with frontal slices ${\beta}^{(k)}_0,{\beta}^{(k)}_1,\ldots,{\beta}^{(k)}_{k-1}$. Or equivalently, we need to find the solution of the following system of tensor equations:
	 	 \begin{align}\label{sysofequa35}
	 	 \begin{cases}
	 	 -\Theta_1 \ast \beta_{0}^{(k)}+ \Theta_2 \ast \beta_{1}^{(k)}  &=\mathscr{O}  \\
	 	 ~~~~~~~~~~~~~~~~~~  -\Theta_2 \ast \beta_{1}^{(k)}+ \Theta_3 \ast \beta_{2}^{(k)}  &=\mathscr{O}  \\
	 	 ~~~~~~~~~~~~~~~~~~ ~~~~~ \ddots \\
	 	 ~~~~~~~~~~~~~~~~~~~~~~~~~~~~~~~~~~~~-\Theta_k \ast \beta_{k-1}^{(k)} &=-\Theta_{k+1} \\
	 	 \end{cases},	
	 	 \end{align}
	 	 here $\mathscr{O}$ stands for zero tensor of order $n_2\times n_2\times n_3$. It is immediate to see that
	 	 \begin{equation}\label{Eq39}
	 	 \beta_{i}^{(k)}=(\Theta _{i+1} )^{-1}\ast\Theta_{k+1} \quad 0 \leq i<k,
	 	 \end{equation}
	 	 is a solution of \eqref{sysofequa35}. Evidently, we have 
	 	 \begin{equation}\label{Eq40}
	 	 \sum_{i=0}^{k}\beta_{i}^{(k)} =\sum_{i=0}^{k} (\Theta_{i+1} )^{-1}\ast\Theta_{k+1},
	 	 \end{equation}
	 	\noindent where $\beta_{k}^{(k)}=\mathscr{I}_{n_2n_2n_3}$. By the discussions in Subsection \ref{sub3.1}, from Eq. \eqref{37}, we can derive 
	 	 $\gamma_{j}^{(k)}$ for $j=0,1,\ldots,k-1$ noticing that $\gamma_{k}^{(k)}=\mathscr{I}_{n_2n_2n_3}-\sum_{i=0}^{k-1}\gamma_{i}^{(k)}$.
	 	 For $i=0,1,\ldots,k-1$, we can further compute $\alpha_i^{(k)}$ for $i=0,1,\ldots,k-1$ by \eqref{scalairesalpha}.
	 	
	 	 \noindent Finally, the extrapolated  third tensor  can be written as follows:
	 	 \begin{align}\label{alphak}
	 	 {T}_{k} = \Delta{\mathscr{S}_{k}} \star \alpha_{}^{(k)}
	 	 \end{align}
	 	 where $\Delta {\mathscr{S}_{k}}$ and $\alpha_{}^{(k)} $ are 4-mode tensors
	 	 whose $j$-th frontal slices are receptively given by   $\Delta {S}_{j-1}={S}_{j}-{S}_{j-1}= \bar{V}_{j}\ast\delta _{j}$ and
	 	 $\alpha_{j-1}^{(k)}$
	 	 for
	 	 $j=1,2,\ldots,k$.
	 	
	 	 \noindent The generalized residual can be written in the following form
	 	 \[
	 	 {R}({T} _{k})= \Delta {\mathscr{S}_{k}} \star \gamma^{(k)}=\sum_{i=0}^{k}\bar{V}_{i+1}\ast \delta_{i+1} \ast \gamma_i^{(k)},
	 	 \]
	 	 where $\gamma^{(k)}$ is a 4-mode tensor with frontal slices $\gamma_0^{(k)},\gamma_2^{(k)},\ldots,\gamma_{k-1}^{(k)}$.
	 	 By Remark \ref{rem4.1}, one can derive
	 	 \begin{align*}
	 	 {R}({T} _{k})^T \ast {R}({T} _{k}) &= \sum_{i=0}^{k}(\gamma_{i}^{(k)})^{T}\ast \delta _{i+1}^{T}\ast \delta _{i+1}\ast\gamma_{i}^{(k)} \\
	 	 &=\sum_{i=0}^{k}(\gamma_{i}^{(k)})^{T}\ast \Theta _{i+1}  \ast\gamma_{i}^{(k)}
	 	 \end{align*}
	 	 Now by Eqs. \eqref{37}, \eqref{Eq39} and \eqref{Eq40}, we can observe
	 	 \begin{align*}
	 	 {R}({T} _{k})^T \ast {R}({T} _{k}) &= \Theta _{k}  \ast\gamma_{k-1}^{(k)}.
	 	 \end{align*}
	 	 It is known that (see \cite{B0290})
	 	 \begin{equation}\label{Eq44}
	 	 \|{R}({T} _{k})\|^2 = {\rm trace} \left( \left(\Theta _{k}  \ast\gamma_{k-1}^{(k)}\right)_{::1}\right).
	 	 \end{equation}
	 	 For ill-posed problems, the value of $\|{R}({T} _{k})\|$ decrease when $k$ increases and is sufficiently small. However, the norm of ${R}({T} _{k})$ may increase with $k$. Hence, similar to  \cite{B366}, we may need to exploit an alternative stopping criterion when the problem is ill-posed.
	 	 To this end, we can use
	 	 \begin{equation}\label{eta}
	 	 \eta_k:= \frac{\| {T}_{k+1}-{T}_{k}\|}{\|\widetilde{T}_{k}\|}=\frac{\sqrt{{\rm trace}\left( \left(({T}_{k+1}-{T}_{k})^T\ast(\bar{T}_{k+1}-{T}_{k})\right)_{::1}\right)}}{ \sqrt{{\rm trace} \left( \left({T}_{k}^T\ast {T}_{k}\right)_{::1}\right)}}.
	 	 \end{equation}
	 	 Using Remark \ref{rem4.1} and some computations, one may simplify the above relation in the following way,
	 	 \[
	 	 ({T}_{k+1}-{T}_{k})^T\ast({T}_{k+1}-{T}_{k})=\sum_{j=1}^{k} (\alpha_{j-1}^{(k+1)}-\alpha_{j-1}^{(k)})^T\ast \Theta_j \ast (\alpha_{j-1}^{(k+1)}-\alpha_{j-1}^{(k)})+\alpha_k^{(k+1)}\ast \Theta_{k+1}\ast \alpha_k^{(k+1)},
	 	 \]
	 	 and
	 	 \[
	 	 {T}_{k}^T\ast {T}_{k}=\sum_{j=1}^{k} (\alpha_{j-1}^{(k)})^T \ast \Theta_j \ast \alpha_{j-1}^{(k)}.
	 	 \]
	 	
	 	 \noindent We end this section by summarizing the above discussions in Algorithm \ref{BTAA}.
	 	
	 	 \begin{algorithm}[H]
	 	 	\caption{The TRRE-TSVD Algorithm}\label{BTAA}
	 	 	{\bf Input.} $\mathscr{A},\mathscr{B}\in \mathbb{R}^{n_1\times n_2 \times n_3}$, tolerance  $\epsilon$  and $k\le \min(n_1,n_3)$ for TTSVD of $\mathscr{A}$;\\
	 	 	\noindent Step 1. Run Algorithm \ref{TSVD} to compute an approximation  for its Moore-Penrose inverse, i.e., $\mathscr{A}^\dagger=\mathscr{V}_{}\ast \mathscr{S}_{}^\dagger \ast \mathscr{U}^T_{}$.\\
	 	 	\noindent Step 2. Set $\bar{  S}_{0}=\mathscr{O}, ~\bar{S}_{1}=\bar{V}_{ 1}  \ast    d^{\dagger}_{1}\ast \bar{U}_{1}^{T} \ast   \bar{ \mathscr{B}}$,  $\bar{T}_{1}=\bar{  S}_{1}$ and $k=2$.\\
	 	 	\noindent Step 3.  ${\rm tol}=1$.\\
	 	 	\noindent Step 4.  while ${\rm tol}< \epsilon$\\
	 	 	$~~~~~~~~~~~~~~~~~~~$Compute $\bar{  S}_{k}$ from  (\ref{ttsvdtrre})\\ 
	 	 	$~~~~~~~~~~~~~~~~~~~$Compute  $ \gamma_{i}^{(k)}$ and $\alpha_{i}^{(k)}$ for $i= 0,\ldots,k-1$ using (\ref{37}) and \eqref{scalairesalpha} where\\ $~~~~~~~~~~~~~~~~~~~~~\gamma_{k}^{(k)}=\mathscr{I}_{n_2n_2n_3}-\sum_{i=0}^{k-1}\gamma_{i}^{(k)}$.\\
	 	 	$~~~~~~~~~~~~~~~~~~~$Compute the approximation $\bar{T}_{ k}$ using  (\ref{alphak}).\\
	 	 	$~~~~~~~~~~~~~~~~~~~~~$Compute $\|\tilde{R}(\bar{T} _{k})\|$ (cf. \eqref{Eq44})  and $\eta_k$ (cf. \eqref{eta}).\\
	 	 	$~~~~~~~~~~~~~~~~~~~~~$${\rm tol}=\min(\|\tilde{R}(\bar{T} _{k})\|,\eta_k)$.\\
	 	 	$~~~~~~~~~~~~~~~~~~~~~$$k=k+1$.\\
	 	 	$~~~~~~~~~~~~$end
	 	 \end{algorithm}

\section{Conclusion} We proposed extrapolation methods in tensor structure. The first class contains the tensor extrapolation polynomial-type methods while for the second class, we introduced the tenasor topological $\epsilon$-transformations. These techniques can be regarded as generalizations of the well-known vector and matrix extrapolation methods. Besides, we introduced some new products between two tensors which can simplify the derivation of  extrapolation methods based on tensor format.
Some theoretical results were also established including the properties of introduced tensor products and an expression for the minimum norm least-square solution of a tensor equation. Finally, the proposed technique was applied on the sequence of tensors corresponding to the truncated tensor singular value decomposition which can be used for solving tensor ill-posed problems.

	 \vspace{0.3cm}

 \end{document}